\documentclass[doublecolumn]{IEEEtran}

\usepackage{tikz, calc}
\usetikzlibrary{shapes,arrows}
\usepackage{pgfplots}
\usepackage{theorem}
\usepackage{tikz}
\usepackage{graphicx}
\graphicspath{{figures/}}
\usepackage{amsfonts,amsmath,amssymb}
\usepackage[space]{cite}
\usepackage{algorithm}
\usepackage{algorithmic}
\usepackage{multirow}
\usepackage{rotating}
\usepackage{subfigure}
\usepackage{url}
\usepackage{wrapfig,verbatim}

\DeclareMathOperator*{\argmin}{arg\,min}
\DeclareMathOperator*{\argmax}{arg\,max}

\newtheorem{thm}{\bf \noindent Theorem}

\newtheorem{lem}{\bf \noindent Lemma}
\newtheorem{prop}{\bf \noindent Proposition}

\newcommand{\R}{\mathbb R}

\newcommand{\todo}[1]{}

\newcommand{\bsd}{\boldsymbol{\delta}}
\newcommand{\tbx}{\tilde{\mathbf{x}}}
\newcommand{\tbu}{\tilde{\mathbf{u}}}

\newcommand{\bl}{\boldsymbol{\lambda}}
\newcommand{\bx}{ {\bf x} }
\newcommand{\hbx}{\hat{\mathbf x}}
\newcommand{\by}{ \tilde{\bf x} }
\newcommand{\bs}{ {\bf s} }
\newcommand{\br}{ {\bf r} }
\newcommand{\bb}{ {\bf b} }
\newcommand{\ba}{ {\bf A} }

\newcommand{\X}{ \mathcal{X} }

\newcommand{\I}{ \mathcal{I} }
\newcommand{\Fc}{ \mathcal{F} }

 \newcommand{\suml}{\sum\nolimits}

\newcommand{\U}{\mathcal{U}}
\newcommand{\V}{\mathcal{V}}

\newcommand{\C}{\mathcal{C}}
\newcommand{\TC}{\tilde{\mathcal{C}}}

\newcommand{\bu}{\mathbf{u}}

\newcommand{\bB}{\mathbf{B}}

\def\0b{\mathbf{0}}

\def\bs{\mathbf{s}}

\newcommand{\la}{\langle}
\newcommand{\ra}{\rangle}
\newcommand{\sml}[1]{{#1}}

\newcommand\Item[1][]{%
  \ifx\relax#1\relax  \item \else \item[#1] \fi
  \abovedisplayskip=0pt\abovedisplayshortskip=0pt~\vspace*{-\baselineskip}}

\allowdisplaybreaks

\IEEEoverridecommandlockouts

\begin{document}

\title{Complexity Certification of a Distributed Augmented Lagrangian Method}

\author{Soomin~Lee*,~\IEEEmembership{Member,~IEEE,}
Nikolaos~Chatzipanagiotis,~\IEEEmembership{Member,~IEEE,}
and~Michael~M.~Zavlanos,~\IEEEmembership{Member,~IEEE}
\thanks{Soomin Lee is with the Dept. of Industrial and Systems Engineering, Georgia Tech, Atlanta, GA, 30318, USA, {\tt\footnotesize soomin.lee@isye.gatech.edu}.
Nikolaos Chatzipanagiotis and Michael M. Zavlanos are with the Dept. of Mechanical Engineering and Materials Science, Duke University, Durham, NC, 27708, USA, {\tt\footnotesize \{n.chatzip,michael.zavlanos\}@duke.edu}. This work is supported by NSF under grant CNS \#1261828, and by ONR under grant \#N000141410479.
}}

\maketitle

\begin{abstract}
In this paper we present complexity certification results for a distributed Augmented Lagrangian (AL) algorithm
used to solve convex optimization problems involving globally coupled linear constraints.
Our method relies on the Accelerated Distributed Augmented Lagrangian (ADAL) algorithm,
which can handle the coupled linear constraints in a distributed manner based on local estimates of the AL.
We show that
the theoretical complexity of ADAL
to reach an $\epsilon$-optimal solution both in terms of suboptimality and infeasibility is $O(\frac{1}{\epsilon})$ iterations. Moreover, we provide a valid upper bound for the optimal dual multiplier which enables us to explicitly specify these complexity bounds.
We also show how to choose the stepsize parameter to minimize the bounds on the convergence rates.
Finally, we discuss a motivating example, a model predictive control (MPC) problem, involving a finite number of subsystems which interact with each other via a general network.
\end{abstract}

\begin{IEEEkeywords}
Augmented Lagrangian methods, computational complexity,
distributed model predictive control.
\end{IEEEkeywords}

\section{Introduction}\label{sec_introduction}

\IEEEPARstart{D}{istributed} optimization methods
decompose large-scale problems into more manageable subproblems that can be efficiently solved in parallel. Moreover, distributed algorithms allow for better load balancing among the available computational resources (inexpensive devices or subsystems) and they also
alleviate drawbacks of centralized systems, such as the cost, fragility, and privacy associated with centralized coordination.
For this reason, they  are widely  used to solve  large-scale problems arising in areas as diverse as  optimal control, wireless communications,  machine learning, computational biology, finance and statistics, to name a few.

Classic decomposition algorithms utilize the separable structure of the dual function.
These methods have low computational cost, but they suffer from slow convergence due to the non-differentiability of the dual functions induced by the ordinary Lagrangian \cite[Chapter 2.6]{Berts1}.
Although this drawback can be avoided by using the Augmented Lagrangian (AL) framework \cite[Chapter 2.1]{Bert_Constrained},
AL based methods lose the decomposable structure of the ordinary Lagrangian,
which makes distributed computation difficult.
This calls for the development of specialized AL decomposition techniques.

Early specialized techniques that allow for decomposition of the AL can be traced back to the works \cite{Tatjewski,Watanabe,Teboulle}. More recent literature involves the \emph{Diagonal Quadratic Approximmation} (DQA) algorithm  \cite{Mulvey,Rus} and the \emph{Alternating Direction Method of Multipliers} (ADMM) \cite{Eck_DR,Eck_Monotrop,Boyd_ADMM}.
The DQA method replaces each minimization step in the augmented Lagrangian algorithm   by a separable approximation of the AL function. The ADMM methods are based on the relations between splitting methods for monotone operators, such as Douglas-Rachford splitting, and the proximal point algorithm  \cite{Glow,Eck_DR}.
Recently, the convergence rate of ADMM has been studied extensively;
see e.g. \cite{WYin2015} and references therein.
Most of these results assume either smoothness, strong convexity, or
strict convexity of the objective function.
Although the results in \cite{He2012,He2015} do not require such properties, the convergence rates are given either in terms of the violation of optimality conditions \cite{He2012}  or the relative change in consecutive iterates \cite{He2015}.

The contributions of this paper are the following:

\noindent 1)  We revisit the general purpose AL method ADAL, first developed for convex optimization problems \cite{Nikos_math_prog,Nikos_ACC2015} and later extended to non-convex problems \cite{Nikos_nonconvexADAL} and problems with noise \cite{Nikos_SADAL}, which relies on local estimates of the AL to handle globally coupled linear constraints in a distributed manner.
We provide computational complexity certifications
for the ADAL method in terms of primal suboptimality and primal infeasibility.
Specifically, we show that
the number of iterations
to reach an $\epsilon$-optimal and $\epsilon$-feasible solution
is $O(\frac{1}{\epsilon})$, under the assumption that the objective function is generally convex and not necessarily differentiable.
This analysis can benefit many practical applications, such as model predictive control (MPC), one of the most successful control frameworks implemented on embedded systems.
As the sampling times for embedded systems are very short,
any iterative optimization algorithm implemented on such systems must be able to precondition the execution time  by providing an explicit number of iterations needed to obtain a reasonably good solution in terms of suboptimality and infeasibility.
For this reason, there has been a growing interest recently in enhancing MPC methods by providing the worst-case computational complexity \cite{RichterConf,RichterTAC,PatrinosTAC,PontusThesis,QuocInexact}.

\noindent 2) Since the complexity bounds above depend on the optimal dual multiplier $\bl^*$, we provide a valid upper bound for $\bl^*$.
Our bound holds for any general convex problems with Lipschitz gradients involving linear constraints.
Tighter bounds for quadratic problems have been studied in \cite{RichterConf,RichterTAC,PatrinosTAC}.

\noindent 3) We show how to select the algorithm parameter $\rho$, which is the stepsize used in the dual gradient step.
To the best of our knowledge, such parameter selection has been  studied
only when the objective function is quadratic or has special properties like strong convexity and smoothness \cite{PontusMetric,admmParam}.


\section{Accelerated Distributed AL\label{sec_alg}}
This section describes the Accelerated Distributed Augmented Lagrangian (ADAL for short)
method, a specialized \emph{Augmented Lagrangian} (AL) decomposition technique which was proposed in \cite{Nikos_math_prog},
for solving optimization problems of the form:
\begin{align}\label{P}
\min_{\bx_i}  ~& ~ \sum\nolimits_{i=1}^N f_i(\bx_i) \notag\\
\text{subject to } ~&~ \sum\nolimits_{i=1}^N \ba_i \bx_i = \bb, \\
 &~~ {\bx}_i \in \mathcal{X}_i,\quad i=1,2,\dots,N,\notag
\end{align}
where $\bx_i \in \mathbb{R}^{n_i}$ denotes the decision variables that belong to subsystem $i$,
and $f_i:\mathbb{R}^{n_i} \to \mathbb{R}$ is its local objective function.
Problem \eqref{P} models situations where a set $\mathcal{I}=\{ 1,2,\dots,N\}$ of decision makers, henceforth referred to as agents, need to determine local decisions $\bx_i\in\mathcal{X}_i$ that minimize the summation of the local functions $f_i(\bx_i)$, while respecting a set of affine coupling constraints $\sum_{i=1}^N \ba_i \bx_i = \bb$.
Here, we assume the functions $f_i: \mathbb{R}^{n_i}\rightarrow \mathbb R$ are convex (not necessarily differentiable)
for all $i\in\mathcal{I}$, the local sets $\mathcal{X}_{i}\subseteq \R^{n_i}$ for $i\in\mathcal{I}$ are convex, closed and bounded, $\ba_i \in \mathbb{R}^{m\times n_i}$, $\bb \in \mathbb{R}^m$, and $n=\sum_{i=1}^N n_i$.

Furthermore, we let
\[
F({\bx}) :=  \sum\nolimits_{i=1}^{N} f_i({\bx}_i),
\]
where $\mathbf{x} = [ \bx_1^\top,\dots, {\bx}_N^\top]^\top \in \mathbb{R}^n$. Denoting $\ba = [ \ba_1 \dots \ba_N ] \in\R^{m\times n}$,  the constraint $\sum_{i=1}^{N} {\ba}_i {\bx}_i = {\bb}$ in problem \eqref{P} becomes $\ba \bx = \bb$.
Also, we define the maximum degree $q$ as a measure of sparsity of the matrix $\ba$, i.e.,
for each constraint $j = 1, \ldots, m$, we denote by $q_j$ the number of all $i \in \I$ such that $[\ba_i]_j \neq \0b$, where $[\ba_i]_j$ is the $j$-th row of matrix $\ba_i$ and $\0b$ stands for a vector of all zeros.
Then, $q$ is defined as:
\begin{align}\label{eqn:q}
q = \max_{j = 1,\ldots,m} q_j.
\end{align}
It will be shown below that $q$ plays a critical role in the convergence properties of the
proposed method.

\subsection{Preliminaries: AL Framework}
Associating Lagrange multipliers $\bl\in\R^m$ with the affine constraint $\ba \bx = \bb$, the Lagrangian for \eqref{P}  is defined as
\begin{align}\label{ord_lagr}
L(\bx,\bl) &= F(\bx) + \langle \bl, \ba \bx -\bb\rangle
= \sum\nolimits_{i=1}^{N} L_i (\bx_i,\bl) -\langle \bb,\bl\rangle,
\end{align}
where $L_i (\bx_i,\bl) = f_i({\bf x}_i) + \langle  \bl , {\ba}_i {\bx}_i\rangle$, and $\langle\cdot,\cdot\rangle$ denotes inner product. Then, the dual function is defined as
\begin{equation}\label{dual_functional}
g(\bl) = \inf_{\bx\in\mathcal{ X}} ~L(\bx,\bl) = \sum\nolimits_{i=1}^{N} g_i(\bl) -\langle \bb,\bl\rangle,
\end{equation}
where $\mathcal{X}= \mathcal{X}_{1}\times\mathcal{X}_{2}\dots\times\mathcal{X}_{N}$, and
\begin{equation*}\label{dual_functionals}
g_i(\bl) = \inf_{\bx_i\in\mathcal{ X}_i} \Big[  f_i({\bf x}_i) + \langle  \bl , {\bf A}_i {\bf x}_i\rangle \Big].
\end{equation*}
The dual function is decomposable with respect to $\bx_i$'s and this gives rise to decomposition methods that address the  \emph{dual problem} \cite[Chapter 2.6]{Berts1}
\begin{equation}\label{D}
\max_{\bl\in\mathbb{R}^m} \, \sum\nolimits_{i=1}^{N} g_i(\bl) -\langle \bb,\bl\rangle.
\end{equation}

Such dual methods suffer from well-documented disadvantages, the most notable one being their exceedingly slow convergence rates due to the nondifferentiability of the dual function \eqref{dual_functional}. These drawbacks can be alleviated by the AL framework \cite[Chapter 2.1]{Bert_Constrained}.
The AL is obtained by adding a quadratic penalty term to the ordinary Lagrangian.
The AL associated with problem \eqref{P} is
\begin{align}\label{augm_lagr}
\Lambda_{\rho}(\bx,\bl) ~&=~ F(\bx) ~+~ \langle \bl , {\bf A} {\bx} - \bb\rangle ~+~ \frac{\rho}{2}\| {\bf A}\bx - \bb \|^2,
\end{align}
where $\rho>0$ is a penalty parameter. We recall that the standard Augmented Lagrangian method is also referred to as the \emph{Method of Multipliers} in the literature \cite[Chapter 2.1]{Bert_Constrained}. 
A major drawback of the Augmented Lagrangian Method stems from the fact that \eqref{augm_lagr} is not  separable with respect to each $\bx_i$ due to the additional quadratic penalty term.

\subsection{The ADAL Algorithm}

The lack of decomposability of the AL calls for the development of specialized AL decomposition techniques.
ADAL is a primal-dual iterative method utilizing a local AL function $\Lambda_{\rho}^i$ which is defined as:
\begin{align}\label{local_lagr}
\Lambda_{\rho}^i(\bx_i,\bx_{-i}^k,\bl) ~=&~ f_i(\bx_i) ~+~ \langle \bl ,{\bf A}_i\bx_i \rangle \\
&\qquad ~+~ \frac{\rho}{2}\| {\bf A}_i\bx_i + \sum\nolimits_{j\in\mathcal{I}}^{j\neq i} {\bf A}_j\bx_j^k - \bb\|^2, \notag
\end{align}
where $\bx_{-i}^k = [\bx_1^k, \ldots, \bx_{i-1}^k,\bx_{i+1}^k,\ldots,\bx_N^k]^{\top}$.
%
The ADAL method is summarized in Alg. \ref{ADAL}.
ADAL has two parameters: a positive penalty parameter $\rho$ and a stepsize parameter
$\tau \in (0,1/q)$.
Each iteration of ADAL consists of three steps: i) every agent solves a local subproblem in a parallel
fashion based on the local approximation of the AL in \eqref{local_lagr}; 
ii) the agents update and communicate their primal variables to neighboring agents;
and iii) they update their dual variables based on the values of the communicated primal variables.

We emphasize here that the quantities ${\bf A}_j\bx_j^k$, appearing in the penalty term of the local AL \eqref{local_lagr}, correspond to the local primal variables of agent $j$ that are communicated to agent $i$. With respect to agent $i$, these are considered fixed parameters. The penalty term of each $\Lambda_{\rho}^i$ can be equivalently expressed as
\begin{align*}
&\| {\bf A}_i\bx_i + \sum\nolimits_{j\in\mathcal{I}}^{j\neq i} {\bf A}_j\bx_j^k - \bb \|^2 ~= \\
&  \qquad\qquad =~ \sum\nolimits_{l=1}^m \Big(  \big[{\bf A}_i\bx_i\big]_l + \sum\nolimits_{j\in\mathcal{I}}^{j\neq i} \big[{\bf A}_j\bx_j^k\big]_l - b_l \Big)^2.
\end{align*}
The above penalty term is present only in the  minimization computation \eqref{ADAL_1}, in Alg. \ref{ADAL}. Hence, for those $l$ such that $[{\bf A}_i]_l = {\mathbf 0}$, the terms $\sum_{j\in\mathcal{I}}^{j\neq i} \big[{\bf A}_j\bx_j^k\big]_l -b_l$  are just constant terms in the minimization step, and can be neglected. Here, $[{\bf A}_i]_l$ denotes the $l$-th row of $\ba_i$  and ${\bf 0}$ stands for a zero vector of proper dimension. This implies that agent $i$ needs access only to the decisions $ \big[{\bf A}_j\bx_j^k\big]_l$ from all agents $j\neq i$ that are present in the same constraints $l$ as $i$. Moreover, regarding the term $\langle \bl ,{\bf A}_i\bx_i \rangle$ in \eqref{local_lagr}, we have that $\langle \bl ,{\bf A}_i\bx_i \rangle ~=~ \sum_{j=1}^m \lambda_j [{\bf A}_i\bx_i]_j$. Hence, we see that, in order to compute \eqref{ADAL_1}, each agent $i$ needs access only to those $\lambda_j$ for which $[{\bf A}_i]_j\neq {\mathbf 0}$.

\begin{algorithm}[t]\caption{Accelerated Distributed Augmented Lagrangians (ADAL)}\label{ADAL}
Set $k=0$, $\tau \in (0,\frac{1}{q})$ and define initial Lagrange multipliers $\bl^0$ and initial primal variables $\bx^0$.
\begin{enumerate}
\item[1.] For fixed Lagrange multipliers $\boldsymbol{\lambda}^k$, determine $\hbx_i^k$ for every $i\in\mathcal{I}$ as the solution of the following problem:
\begin{equation}\label{ADAL_1}
\begin{aligned}
\min_{\bx_i \in \mathcal{X}_i}\, &\, \Lambda_{\rho}^i(\bx_i,\bx_{-i}^k,\bl^k).
\end{aligned}
\end{equation}
\item[2.] Set for every $i\in\mathcal{I}$
\begin{equation}\label{ADAL_2}
\bx_i^{k+1} = \bx_i^k + \tau (\hbx_i^k -  \bx_i^k ).
\end{equation}
\item[3.] If the constraints $\sum_{i=1}^{N}{\bf A}_i \bx_i^{k+1} = \bb $ are satisfied and $\ba_i\hat{\bx}_i^{k} = \ba_i\bx_i^{k}$ for all $i\in\mathcal{I}$,
then stop (optimal solution found). Otherwise, set:
\begin{equation}\label{ADAL_3}
\bl^{k+1} = \bl^k + \rho\tau \Big(  \sum\nolimits_{i=1}^{N} {\bf A}_i \bx^{k+1}_i - \bb\Big),
\end{equation}
increase $k$ by one and return to Step 1.
\end{enumerate}
\end{algorithm}

\subsection{A Motivating Example: Distributed Model Predictive Control (DMPC) with linear coupling constraints\label{subsec:mpc}}
Consider a discrete-time linear dynamical system expressed in terms of the dynamics of
a set $\I=\{1,\dots,N\}$ of individual subsystems as
\begin{align}\label{local}
&\bx_i^{t+1} = \suml_{j\in\C_i^t} \left(\ba_{ij}^t \bx_j^t + \bB_{ij}^t \bu_j^t \right)	\nonumber\\
&\bx_i^{t}\in\X_i^{t}, \quad \bu_i^t\in\U_i^t, \quad \forall ~i\in\I,
\end{align}
where  $\bx_i^t\in\X_i^t\subseteq \R^{n_i}$ and $\bu_i^t\in\U_i^t\subseteq\R^{p_i}$ represent a local state and input at time $t$. We assume that the local constraint sets $\X_i^t,~\U_i^t$ satisfy $\X^t = \X_1^t \times \cdots \times \X_N^t$, $\U^t = \U_1^t \times \cdots \times \U_N^t$, and $n=\sum_{i\in\I} n_i$, $p=\sum_{i\in\I} p_i$.
The dynamic interconnections at time $t$ among the subsystems are modeled by a directed graph $\mathcal{G}^t=(\mathcal{I},\mathcal{E}^t)$. The set of edges $\mathcal{E}^t\subseteq \mathcal{I}\times\mathcal{I}$ contains a directed edge $(v_i,v_j)$ if the state or input of subsystem $i$ at time $t$ affects the dynamics of subsystem $j$ at time $t+1$. More formally, $(v_j,v_i)\in \mathcal{E}^t$ if and only if  $\ba_{ij}^t \neq 0 ~\vee~ \bB_{ij}^t \neq 0$, where the matrices $\ba_{ij}^t \in\R^{n_i\times n_j}$ and $\bB_{ij}^t\in\R^{n_i\times p_j}$, define the dynamic coupling between subsystems $i$ and $j$ at time $t$.
We define the coupling in-neighborhood $\mathcal{C}_i^t$ (resp. out-neighborhood $\TC_i^t$) of subsystem $i$ at time $t$ as the set of sybsystems $j$ whose dynamics at $t$ affect (resp. is affected by) the evolution of subsystem $i$, i.e.,  $\C_i^t = \{ j\in\I : (v_j,v_i)\in \mathcal{E}^t \}$ (resp. $\TC_i^t = \{ j\in\I : (v_i,v_j)\in \mathcal{E}^t \}$).

Determining optimal control sequences for \eqref{local} using MPC consists of  solving online a finite horizon open-loop optimal control problem, subject to the aforementioned system dynamics and constraints that involve states and control inputs.
Specifically,  the MPC problem for the dynamical system \eqref{local} is parametric to the initial state $\bx^1$ and can be formulated as
\begin{align}\label{DMPC1}
\min_{\bx,\bu} &	~~\sum_{i=1}^N \bigg[ \sum_{t=1}^{H-1} \ell^t_i(\bx_i^t,\bu_i^t) + \Fc_i(\bx_i^{H}) \bigg] \nonumber\\
\text{s.t.} \quad &\bx_i^{t+1} = \suml_{j\in\C_i^t} \left(\ba_{ij}^t \bx_j^t + \bB_{ij}^t \bu_j^t \right),	\\
&\bx_i^{t+1}\in\X_i^{t+1}, \quad \bu_i^t\in\U_i^{t},\nonumber\\
& \forall ~i\in\I ~\text{and} ~t\in\{1,\dots,H-1\}.  \nonumber
\end{align}
where the functions  $\ell^t_i(\bx_i^t,\bu_i^t): \R^{n_i} \times \R^{p_i} \to \R$ denote the running cost and the function $\Fc_i(\bx_i^H): \R^{n_i}\to \R$ denotes the terminal cost of subsystem $i$.

To use the ADAL framework in Alg. \ref{ADAL} to solve \eqref{DMPC1}, we introduce a local AL for each subsystem $i$ as
\begin{align}\label{local_AL}
	&\Lambda_{\rho}^i(\bx_i,\bu_i,\bl) ~=~   \sum_{t=1}^{H-1} \ell^t_i(\bx_i^t,\bu_i^t) ~+ \Fc_i(\bx_i^{H})    \\
	& +\sum_{t=1}^{H-1} \Bigg[ (\bl_i^{t+1})^T\bx_i^{t+1} -  \sum_{j\in\TC_i^t} (\bl_j^{t+1})^T \left(\ba_{ji}^t \bx_i^t + \bB_{ji}^t \bu_i^t\right)  \nonumber\\
	& + \frac{\rho}{2}\| \bx_i^{t+1} - \ba_{ii}^t \bx_i^t - \bB_{ii}^t \bu_i^t -\sum_{j\in\C_i^t\backslash\{i\}} \left(\ba_{ij}^t \tbx_j^t + \bB_{ij}^t \tbu_j^t \right) \|^2  \nonumber\\
	& + \sum_{j\in\TC_i^t}\frac{\rho}{2}\| \tbx_j^{t+1} - \ba_{ji}^t \bx_i^t - \bB_{ji}^t \bu_i^t \nonumber\\
	& \qquad \qquad \qquad\qquad- \sum_{m\in\C_j^t\backslash\{i\}}\left(\ba_{jm}^t \tbx_m^t + \bB_{jm}^t \tbu_m^t\right) \|^2 \Bigg],  \nonumber
\end{align}
where  $\tbx_j,\tbu_j$ denote the primal variables that are controlled by subsystem $j$  but communicated to subsystem $i$ for optimization of its local Lagrangian $\Lambda_{\rho}^i$. With respect to subsystem $i$, these are just considered as fixed parameters.
That is, the local AL is created by taking all the terms involving $\bx_i$ in the original AL and setting the remaining variables as fixed parameters, i.e., $\bx_j$ as $\tilde{\bx}_j$ for all $i \neq j$.

Observe that the local AL \eqref{local_AL} of each subsystem $i$ includes only locally available information. Regarding  the dual variables, the necessary information includes $\bl_i^{t+1}$ and  all $\bl_j^{t+1}$ for every $t\in\{1,\dots,H-1\}$ and $j\in \TC_i^t$, i.e., the dual variables corresponding to the dynamical constraints of $i$ and also those of the out-neighbors of subsystem $i$ in all coupling graphs $\mathcal{E}^t$. Regarding the primal variables, the necessary information for the local AL of subsystem $i$ includes all $\tbx^t_j,~\tbu^t_j$ for every $t\in\{1,\dots,H\}$ from the in-neighbors $j\in\C_i^t$, the out-neighbors $j\in\TC_i^t$, and  the in-neighbors  of the out-neighbors of $i$, namely $\{m\in\I: m\in\C_j^t, ~\forall j\in\TC_i^t\}$ for all the coupling graphs $\mathcal{E}^t$. In other words, each subsystem $i$ needs to be able to exchange messages with all subsystems $j$ that belong to its 2-hop communication neighborhood $\I_i = \bigcup_{t=1}^{H}\left( \C_i^t \cup \TC_i^t \cup \{m\in\I: m\in\C_j^t, ~\forall j\in\TC_i^t\} \right)$.

In practice \eqref{DMPC1} is solved repeatedly, and after each solve, the first few inputs are applied to \eqref{local} and the horizon is shifted accordingly, providing a new initial condition for a subsequent solution of \eqref{DMPC1}. In this framework, solving \eqref{DMPC1} until convergence is time consuming. Therefore, early termination is highly desired, while ensuring a good quality solution.

\section{Rate of Convergence\label{sec:rate}}
In this section we characterize the rate of convergence of the ADAL method.
\sml{In what follows, we denote the subgradient of a convex function $f$ at a point $\bx \in \mathcal{X}$ by $\bs_{\bx}$, i.e.,
a vector $\bs_{\bx} \in \mathbb{R}^n$ is a subgradient of $f$ at $\bx \in \mathcal{X}$ if
\begin{align*}
f(\by) \ge f(\bx) + \langle \bs_{\bx}, \by-\bx\rangle,\quad \forall \by \in \mathcal{X}.
\end{align*}
We also denote the convex subdifferential of $f$ at $\bx \in \mathcal{X}$ by $\partial f (\bx)$,
which is the set of all subgradients $\bs_{\bx}$.}

The convergence of ADAL relies on the following three assumptions, which are typically required in the analysis of convex optimization methods:
\begin{itemize}
\item[(A1)] The functions  $f_i$ are convex, and the sets $\mathcal{X}_i$ are convex, closed, and bounded for all $i\in\mathcal{I}$ .
\item[(A2)]
The Lagrangian function $L$ has a saddle point $(\bx^*,\bl^*)\in \R^n\times\R^m$ so that
\begin{equation*}
L(\bx^*,\bl)\leq L(\bx^*,\bl^*)\leq L(\bx,\bl^*),~ \forall  ~\bx\in \X,\,~ \bl\in\R^m.
\end{equation*}
\item[(A3)] All subproblems \eqref{ADAL_1} are exactly solvable at every iteration.
\end{itemize}

Assumption (A1) implies that there exists a constant $D_{\X}$ such that
\begin{align}\label{eqn:DX}
D_{\X} := \max\nolimits_{\bx,\by \in \X} \|\bx-\by\|
\end{align}
and also Lipschitz subgradients, i.e., there exists a constant $G$ such that for all $i \in \I$
\begin{align}\label{eqn:G}
\|\bs_{\bx}\| \le G, \quad \forall \bs_{\bx} \in \partial f_i(\bx),~ \bx \in \X.
\end{align}
Assumption (A2) implies that the point $\bx^*$ is a solution of problem \eqref{P} and the point $\bl^*$ is a solution of \eqref{D}.
Since \eqref{P} is a convex program with linear constraints,
strong duality holds, i.e., the optimal values of the primal and dual problems are equal, as long as \eqref{P} is feasible without the need of any constraint qualification.
Assumption (A3) is satisfied for most MPC problems, see e.g., \cite[Section V]{RichterConf},
or for general problems with simple constraint sets $\X$, e.g., boxes or balls.


\subsection{Lemmas}
In this subsection, we provide a few lemmas that will help us prove the convergence of ADAL.
Our analysis relies on the ergodic average of the primal variables up to iteration $k$:
\begin{align*}
\by^k := \frac{1}{k}\sum\nolimits_{p=0}^{k-1} \hbx^p.
\end{align*}
To avoid cluttering the notation, we will use $\sum_i$ to denote summation over all $i\in\mathcal{I}$, i.e., $\sum_i = \sum_{i=1}^N$, unless explicitly noted otherwise.
We define the \emph{residual} $\br(\bx) \in\mathbb{R}^m$  as the vector containing the amount of all constraint violations with respect to primal variable $\bx$, i.e.,
\begin{align}\label{eqn:residual}
\br(\bx) = \sum\nolimits_i \ba_i \bx_i - \bb.
\end{align}
We also define the auxiliary dual variable $\bar{\bl}^k$ as
\begin{align}\label{lambda_bar}
\bar{\bl}^k ~:=~ \bl^k  +\rho(1-\tau)\br(\bx^k).
\end{align}
In the next lemma, we obtain an iterative relation for $\bar{\bl}^k$.
The proof can be found in \cite[Theorem 1]{Nikos_math_prog}.

\begin{lem}\label{lem:lbar}
The dual update step \eqref{ADAL_3} of ADAL is equivalent to the update rule
\begin{equation*}
\bar{\bl}^{k+1} ~=~ \bar{\bl}^{k}+\tau\rho \br(\hat{\bx}^k).
\end{equation*}
\end{lem}
\vspace{3pt}
%
%

In the next lemma, we utilize Lemma \ref{lem:lbar} and the first order optimality conditions for each local subproblem \eqref{ADAL_1} to
bound the function value at each iteration, which later will allow us to obtain a telescoping sum.
For this, we make use of the Lyapunov/Merit function
\begin{align}\label{phi}
\phi^k(\bl) ~=~ \rho\sum\nolimits_{i}\| \ba_i (\bx_i^k -\bx_i^*) \|^2 + \frac{1}{\rho}\|\bar{\bl}^k-\bl \|^2,
\end{align}
for all $k \ge 0$ and any arbitrary $\bl \in \mathbb{R}^m$.
A similar result whose Lyapunov/Merit function $\phi^k$ does not depend on $\bl$
can be found in \cite{Nikos_ACC2015}. Note that dependence of $\phi^k(\bl)$ on $\bl$ is key to obtain the convergence rates presented in this paper.

%


%
\begin{lem}\label{lemma3}
Assume (A1)--(A3). Then, for any $\bl \in \mathbb{R}^m$ and $k \ge 0$, the following holds:
\begin{align*}
F(\hbx^k) -  F(\bx^*) + \langle \bl, \br(\hbx^k)\rangle
\leq~
\frac{1}{2\tau}\big( \phi^k(\bl) - \phi^{k+1}(\bl) \big).
\end{align*}
\end{lem}
The proof of this lemma can be found in Appendix \ref{app:B}.

\subsection{Primal Optimality and Feasibility}
Using Lemma \ref{lemma3} and the properties of convex functions, we now provide two theorems regarding the convergence rate of ADAL.
More specifically, in Theorem \ref{thm1},
we consider the objective value difference $F(\by^k)-F(\bx^*)$ and the constraint violation $\| \ba \by^k - \bf b\|$ together and show that their sum decreases at a worst-case $O(1/k)$ rate.
In Theorem \ref{thm2}, we upper bound the objective value difference and constraint violation separately, and show that each one of them decreases at a worst-case $O(1/k)$ rate.


\begin{thm}\label{thm1}
Assume (A1)--(A3). Recall that $\by^k = \frac{1}{k}\sum_{p=0}^{k-1} \hbx^p $ denotes the ergodic average of the primal variable sequence generated by ADAL up to iteration $k$
and $\br(\bx) = \ba\bx - \bb$ denotes the residual at $\bx$.
Then, for all $k$
\begin{align}\label{thm_eq}
F(\by^k) -  F(\bx^*) + \|\br( \by^k)\| ~\leq~ \frac{1}{2k\tau} \phi,
\end{align}
where $\phi = \sum\nolimits_{i=1}^N {\rho}\|\ba_i(\bx_i^0-\bx_i^*)\|^2 +\frac{1}{\rho}(\|\bar\bl^0\|+1)^2$.
\end{thm}
%
\begin{proof}
Summing the relation in Lemma \ref{lemma3} for all $p=0,\dots,k-1$, we get
\begin{align}\label{10}
& \sum\nolimits_{p=0}^{k-1}F(\hat{\bx}^p)  ~-~ \sum\nolimits_{p=0}^{k-1} F(\bx^*) ~+~ \sum\nolimits_{p=0}^{k-1}\Big\langle \bl, \br(\hbx^p) \Big\rangle\notag\\
&  \leq~ \frac{1}{2\tau} \Big( \phi^0(\bl) - \phi^{k}(\bl)\Big).
\end{align}
By the convexity of $F$, we have that
\begin{align*}
\sum\nolimits_{p=0}^{k-1}  \frac{1}{k} F(\hat{\bx}^p) ~\geq~  F \Big( \sum\nolimits_{p=0}^{k-1} \frac{1}{k}\hbx^p \Big),
\end{align*}
which implies that $\sum\nolimits_{p=0}^{k-1}F(\hat{\bx}^p) ~\geq~  k F ( \by^k  )$. The analogous relation holds for $\sum_{p=0}^{k-1}\br(\hbx^p) \geq k \br( \by^k  ) $, since it is a linear (convex) mapping. We also have that $\sum_{p=0}^{k-1} F(\bx^*) = kF(\bx^*)$. Hence, \eqref{10} can be expressed as
\begin{align*}
 kF( \by^k  )  - kF(\bx^*) + k\langle \bl, \br( \by^k )\rangle~\leq~ \frac{1}{2\tau} \Big( \phi^0(\bl) - \phi^{k}(\bl)\Big),
\end{align*}
or,
\begin{align}\label{eqn:ergodic}
 F( \by^k  ) - F(\bx^*) + \langle \bl, \br( \by^k  )\rangle  ~\leq~ \frac{1}{2k\tau}   \phi^0(\bl),
\end{align}
because for any $\bl \in \mathbb{R}^m$, we have $\phi^{k}(\bl) \geq 0$.

The above inequality is true for all $\bl \in \mathbb{R}^m$, hence it must also hold for any point in the ball $\mathcal{B} = \{\bl \mid \|\bl\| \le 1\}$. We now let $\bl = \tilde{\bl}^k \triangleq \argmax_{\bl \in \mathcal{B}} \langle \bl, \br( \by^k  )\rangle$ and rewrite the above relation as
\begin{align*}
 F( \by^k  )  - F(\bx^*) +  \|\br( \by^k  )\| ~\leq~ \frac{1}{2k\tau}   \phi^0(\tilde{\bl}^k),
\end{align*}
where we used $\langle \tilde{\bl}^k, \br( \by^k  )\rangle = \|\br( \by^k  )\|$.
Finally, the term on the right-hand side can be bounded as
\begin{align*}
\phi^0(\tilde{\bl}^k) = &\sum\nolimits_{i=1}^N {\rho}\|\ba_i(\bx_i^0-\bx_i^*)\|^2 +\frac{1}{\rho}\|\bar{\bl}^0-\tilde{\bl}^k\|^2\\
\le &\sum\nolimits_{i=1}^N {\rho}\|\ba_i(\bx_i^0-\bx_i^*)\|^2 +\frac{1}{\rho}(\|\bar{\bl}^0\|+1)^2,
\end{align*}
which gives the desired result.
\end{proof}
\vspace{3pt}

The importance of this bound is that
the computation complexity can be specified in advance as long as the diameters of the primal constraint sets $\mathcal{X}_i$ can be determined.
However, when the primal solution $\by^k$ is not feasible, it is possible that $F(\by^k) -F^* < 0$.
In this case, the bound in \eqref{thm_eq} can  still be useful if the primal residual can be tightly bounded
as pointed out in \cite{LanDBnd}, i.e., if $\|\mathbf{A}\by^k-\mathbf{b}\| < \delta$ for a relatively small $\delta >0$,
then a lower bound of $F(\by^k) -F^*$ is given by
\begin{align*}
F(\by^k) -  F(\bx^*)  \ge \langle \boldsymbol{\lambda}^*, \mathbf{A}\by^k-\mathbf{b}\rangle \ge - \delta\|\boldsymbol{\lambda}^*\|,
\end{align*}
where $\boldsymbol{\lambda}^*$ is a component of the saddle point $(\mathbf{x}^*,\boldsymbol{\lambda}^*)$ of \eqref{ord_lagr}.

%
\begin{thm}\label{thm2}
Assume (A1)--(A3). Recall that $\by^k = \frac{1}{k}\sum_{p=0}^{k-1} \hbx^p $ denotes  the ergodic average of the primal variable sequence generated by ADAL up to iteration $k$
and $\br(\bx) = \ba\bx - \bb$ denotes the residual at $\bx$.
Let $(\bx^*,\bl^*)$ be a saddle point of \eqref{ord_lagr}.
Then, for all $k$\vspace{2mm}
\begin{itemize}
\Item[(a)]
\begin{align*}
|F(\by^k) -  F(\bx^*)| \le \frac{1}{2k\tau} \max\{\phi^0(\0b),\phi^0(2\bl^*)\},
\end{align*}
where $\phi^0(\bl) = \sum\nolimits_{i=1}^N {\rho}\|\ba_i(\bx_i^0-\bx_i^*)\|^2 +\frac{1}{\rho}\|\bar{\bl}^0-\bl\|^2$.\\
\Item[(b)]
\begin{align*}
\|\br( \by^k)\|\le &~\frac{1}{2k\tau} \bigg[\sum\nolimits_{i=1}^N {\rho}\|\ba_i(\bx_i^0-\bx_i^*)\|^2 \\ &~+\frac{2}{\rho}\left(\|\bar{\bl}^0-\bl^*\|^2 + 1\right)\bigg].
\end{align*}
\end{itemize}
\end{thm}
%
\begin{proof}
\noindent (a)
The inequality \eqref{eqn:ergodic} is true for all $\bl \in \mathbb{R}^m$, hence
letting $\bl = \mathbf{0}$ yields
\begin{align}\label{eqn:thm1eq0}
 F( \by^k  ) - F(\bx^*)   ~\leq~ \frac{1}{2k\tau}   \phi^0(\mathbf{0}).
\end{align}
Let $\bl^*$ be a dual optimal solution. Then, from the saddle point inequality, we have
\begin{align}\label{eqn:thm1eq1}
F(\bx^*)  ~\leq~  F( \by^k  ) + \langle \bl^*, \br( \by^k  )\rangle,
\end{align}
which implies
\begin{align}\label{eqn:thm1eq2}
F(\bx^*)  - F( \by^k  ) ~\leq~  \langle \bl^*, \br( \by^k  )\rangle.
\end{align}
Next, we find an upper bound of the term $\langle \bl^*, \br( \by^k  )\rangle$.
We add $\langle \bl^*, \br( \by^k  )\rangle$ to both sides of \eqref{eqn:thm1eq1} to obtain
\begin{align*}
\langle \bl^*, \br( \by^k  )\rangle  ~\leq~  F( \by^k  )- F(\bx^*) + \langle 2\bl^*, \br( \by^k  )\rangle.
\end{align*}
Using relation \eqref{eqn:ergodic} again with $\bl = 2\bl^*$ to bound the right-hand side of the above equation, we obtain
\begin{align*}
\langle \bl^*, \br( \by^k  )\rangle  ~\leq~  \frac{1}{2k\tau}   \phi^0(2\bl^*).
\end{align*}
Combining this with relation \eqref{eqn:thm1eq2}, we further obtain
\begin{align*}
F(\bx^*)  - F( \by^k  )  ~\leq~  \frac{1}{2k\tau}   \phi^0(2\bl^*).
\end{align*}
Combining this with relation \eqref{eqn:thm1eq0}, the desired result follows. \\

\noindent (b)
We next bound the residual $\|\br( \by^k  )\|$.
Using relation \eqref{eqn:ergodic} with $\bl = \bl^* + \frac{\br( \by^k  )}{\|\br( \by^k  )\|}$, we have
\begin{align}\label{eqn:thm1eq4}
& F( \by^k  ) - F(\bx^*) + \langle \bl^*, \br( \by^k  )\rangle  + \|\br( \by^k  )\|\nonumber\\
& ~\leq~ \frac{1}{2k\tau}   \phi^0\left(\bl^* + \frac{\br( \by^k  )}{\|\br( \by^k  )\|}\right).
\end{align}
Using the saddle point inequality together with the fact that $(\bx^*,\bl^*)$ is a primal-dual optimal pair,
we obtain
\begin{align*}
F( \by^k  ) + \langle \bl^*, \br( \by^k  )\rangle \ge F(\bx^*) + \langle \bl^*, \br( \bx^*  )\rangle,
\end{align*}
which implies
\begin{align*}
F( \by^k  ) -F(\bx^*) +  \langle \bl^*, \br( \by^k  )\rangle \ge 0.
\end{align*}
Combining this with relation \eqref{eqn:thm1eq4}, we obtain
\begin{align*}
\|\br( \by^k  )\|
 ~\leq~ \frac{1}{2k\tau}   \phi^0\left(\bl^* + \frac{\br( \by^k  )}{\|\br( \by^k  )\|}\right).
\end{align*}
From the definition of the Lyapunov/Merit function $\phi^k(\bl)$ in \eqref{phi},
the right-hand side can be represented as
\begin{align*}
&\phi^0\left(\bl^* + \frac{\br( \by^k  )}{\|\br( \by^k  )\|}\right) \\
&~ = \sum\nolimits_{i=1}^N {\rho}\|\ba_i(\bx_i^0-\bx_i^*)\|^2 +\frac{1}{\rho}\left\|\bar{\bl}^0-\bl^* + \frac{\br( \by^k  )}{\|\br( \by^k  )\|}\right\|^2\\
&~ = \sum\nolimits_{i=1}^N {\rho}\|\ba_i(\bx_i^0-\bx_i^*)\|^2 +\frac{2}{\rho}\left(\|\bar{\bl}^0-\bl^*\|^2 + 1\right),
\end{align*}
from which the desired result follows.
\end{proof}

Theorem \ref{thm2} characterizes the suboptimality and infeasibility of the solution obtained when the algorithm is terminated before reaching the optimal solution. That is, the theoretical complexity
for the algorithm to reach an $\epsilon$-optimal solution both in terms of objective value and feasibility
is $O(\frac{1}{\epsilon})$ iterations.
This result is particularly useful for MPC applications where frequent re-optimization for different time horizons is often required in practice, as discussed in Section \ref{subsec:mpc}.
In order to explicitly specify the complexity in advance, however, these bounds require an estimation on the dual optimal solution $\bl^*$.

\section{Certification of Complexity\label{sec:cert}}
In this section, we provide a valid upper bound for $\bl^*$,
which is a corresponding dual multiplier for the optimal solution $\bx^*$ of problem \eqref{P}.

\begin{thm}\label{thm3}
Assume (A1)-(A3). Let $(\bx^*,\bl^*)$ be a primal-dual optimal pair of \eqref{P} and \eqref{D}.
Then,
\begin{align*}
\|\bl^*\| \le \frac{\sqrt{N}G}{\tilde\sigma_{\min}(\ba)},
\end{align*}
where $\tilde\sigma_{\min}(\ba)$ is the smallest nonzero singular value of $\ba$.
\end{thm}

\begin{proof}
Define a value function $\V:\mathbb{R}^m \to \mathbb{R}$ as
\begin{align*}
\V(\bsd) := \min_{\bx \in \X, \ba\bx=\bb + \bsd} F(\bx).
\end{align*}
By Lagrangian duality, this can be equivalently represented as
\begin{align*}
\V(\bsd) = \max_{\bl\in\mathbb{R}^m} \la \bl, \bb+\bsd\ra + \min_{\bx \in \X} F(\bx) + \la \bl, -\ba\bx \ra.
\end{align*}
Let the function above attain its  value at $ \bl = \bl^*(\bsd)$. Then, $\bl^*(\bsd)\in \partial \V(\bsd)$.
To bound the dual multiplier $\bl^* =\bl^*(\0b)$, therefore, it suffices to show that any vector in $\partial \V(\0b)$ is bounded.
Let $\bs \in \partial \V(\0b)$. Then, from the convexity of $\V(\cdot)$, we have that for any $\epsilon > 0$
\begin{align}\label{L1}
\V\left(\epsilon \frac{\bs}{\|\bs\|}\right) - \V(\0b) \ge \left\la \bs, \epsilon \frac{\bs}{\|\bs\|}\right\ra = \epsilon \|\bs\|.
\end{align}
Let $\bx_{\epsilon}^*$ be defined such that
\begin{align*}
\bx_{\epsilon}^*:= \argmin_{\bx_i \in \X, \ba \bx =\bb+\epsilon \frac{\bs}{\|\bs\|}} F(\bx).
\end{align*}
Then, we have
$\ba(\bx^*-\bx_{\epsilon}^*) = - \epsilon \frac{\bs}{\|\bs\|}$,
 from which we obtain
\begin{align}\label{L2}
\|\bx^*-\bx_{\epsilon}^*\| \le \frac{\epsilon}{\tilde\sigma_{\min}(\ba)},
\end{align}
where $\tilde\sigma_{\min}(\ba)$ is the smallest nonzero singular value of $\ba$.
From \eqref{L2} and \eqref{eqn:G}, we obtain
\begin{align*}
&\left\|\V\left(\epsilon \frac{\bs}{\|\bs\|}\right) - \V(\0b)\right\|
= \|F(\bx_{\epsilon}^*)-F(\bx^*)\| \\
&~~\le G \sum_{i=1}^N\|\bx_{i,\epsilon}^*-\bx^*_i\|
\le  \sqrt{N}G\sqrt{\|\bx_{\epsilon}^*-\bx^*\|^2}
\le \frac{\epsilon\sqrt{N}L}{\tilde\sigma_{\min}(\ba)}.
\end{align*}
In view of this relation and \eqref{L1}, and the fact that $\bs$ represents any arbitrary vector in $\partial \V(\0b)$,
we obtain the desired result.
\end{proof}
\vspace{3pt}

Using the bound above, in the following two propositions,
we provide an explicit number of iterations for the ADAL method
to obtain an $\epsilon$-optimal solution
as well as a selection of the algorithm parameter $\rho$.
Since the bound on the right-hand side of Theorem \ref{thm1} depends on $\bx^*$,
we further upper bound this using relation \eqref{eqn:DX} as
\begin{align}\label{eqn:ubnd1}
\sum\nolimits_{i=1}^N {\rho}\|\ba_i(\bx_i^0\hspace{-0.5mm}-\hspace{-0.5mm}\bx_i^*)\|^2 \hspace{-0.5mm}+\hspace{-0.5mm}\frac{1}{\rho}
\le \rho N\sigma_{\max}^2(\ba)D_{\X}^2 \hspace{-0.5mm}+\hspace{-0.5mm} \frac{1}{\rho},
\end{align}
where we set $\bar \bl^0 = \0b$.

\begin{prop}
Assume (A1)-(A3). Let $\bar \bl^0 = \0b$.
Then, the parameter $\rho^*$ minimizing the bound in \eqref{eqn:ubnd1} is
\begin{align*}
\rho^* = \frac{1}{\sqrt{N}\sigma_{\max}(\ba)D_{\X}}.
\end{align*}
Furthermore, the number of iterations required to decrease the bound \eqref{thm_eq} less than $\epsilon$ is
\begin{align*}
k_{\epsilon,1} = \left\lceil \frac{\sqrt{N}\sigma_{\max}(\ba)D_{\X}}{\epsilon\tau}\right\rceil.
\end{align*}
\end{prop}
\begin{proof}
Note that the right-hand side of relation \eqref{eqn:ubnd1} is convex with respect to $\rho$.
Therefore, it is easy to see that the parameter $\rho$ which minimizes the right-hand side can be chosen as
$\rho^* = \frac{1}{\sqrt{N}\sigma_{\max}(\ba)D_{\X}}$.
By using this parameter for the bound in Theorem \ref{thm1}, we obtain
\begin{align*}
F(\by^k) -  F(\bx^*) + \|\br( \by^k)\| ~\leq~ \frac{1}{k\tau} \sqrt{N}\sigma_{\max}(\ba)D_{\X},
\end{align*}
from which the desired result follows.
\end{proof}
\vspace{3pt}
This result shows that the number of required iterations depends on the number of network agents,
the diameter of the constraint set $\X$,
the maximum singular value of $\ba$, and
the sparsity of the matrix $\ba$, which is encoded in the parameter $\tau \in (0,1/q)$ (cf. Eq. \eqref{eqn:q}).

Similarly, from relation \eqref{eqn:DX} and Theorem \ref{thm3}, the right-hand side of Theorem \ref{thm2}(a), which is larger than that of Theorem \ref{thm2}(b), can be further upper bounded as
\begin{align}\label{eqn:ubnd2}
&\sum\nolimits_{i=1}^N {\rho}\|\ba_i(\bx_i^0-\bx_i^*)\|^2 +\frac{4}{\rho}\|\bl^*\|^2\\
&~~\le \rho N\sigma_{\max}^2(\ba)D_{\X}^2 + \frac{4}{\rho}\frac{NG^2}{\tilde\sigma_{\min}^2(\ba)},\nonumber
\end{align}
where we set $\bar \bl^0 = \0b$.

\begin{prop}
Assume (A1)-(A3). Let $\bar \bl^0 = \0b$.
Then, the parameter $\rho^*$ minimizing the bound in \eqref{eqn:ubnd2} is
\begin{align*}
\rho^* = \frac{2G}{\tilde\sigma_{\min}(\ba)\sigma_{\max}(\ba)D_{\X}}.
\end{align*}
Furthermore, the number of iterations required to obtain an $\epsilon$-optimal and feasible solution is
\begin{align*}
k_{\epsilon,2} =\left\lceil  \frac{2GN\sigma_{\max}(\ba)D_{\X}}{\epsilon\tau\tilde\sigma_{\min}(\ba)}
\right\rceil.
\end{align*}
\end{prop}
\begin{proof}
Since the right-hand side of relation \eqref{eqn:ubnd2} is convex with respect to $\rho$,
it is easy to see that the parameter $\rho$ which minimizes the right-hand side can be chosen as
$\rho^* = \frac{2G}{\tilde\sigma_{\min}(\ba)\sigma_{\max}(\ba)D_{\X}}$.
By using this parameter for the bound in Theorem \ref{thm2}(a), we obtain
\begin{align*}
|F(\by^k) -  F(\bx^*)| ~\leq~ \frac{1}{2k\tau} \frac{4GN\sigma_{\max}(\ba)D_{\X}}{\tilde\sigma_{\min}(\ba)},
\end{align*}
from which the desired result follows.
\end{proof}
\vspace{3pt}

As expected, $k_{\epsilon,2} \ge k_{\epsilon,1}$ since the conditions imposed by Theorem \ref{thm2} are more strict.
More specifically, due to the dependence of the bounds on the optimal dual multiplier $\bl^*$, $k_{\epsilon,2}$ also depends on the Lipschitz constant $G$.

\section{Conclusions}\label{sec_concl}

In this paper, we presented an Augmented Lagrangian decomposition method (ADAL)
and characterized its computational complexity. We showed that the algorithm generates an $\epsilon$-optimal and feasible solution using the ergodic average of the sequence of primal variables under some mild assumptions such as the general convexity of the problems.
We also provided an explicit upper bound on the optimal dual multiplier, from which the number of iterations can be explicitly given for any general convex problems involving linear constraints.
The results in this paper have the potential to significantly improve the performance of distributed MPC problems,
where preconditioning of computational complexity is important.

\appendix

\subsection{Proof of Lemma \ref{lemma3}}\label{app:B}
Let $\bs_i^k$ be a subgradient of $f_i$ at $\hat{\bx}_i^k$, i.e., $ \bs_i^k \in \partial f_i(\hat{\bx}_i^k)$.
Then, the first order optimality conditions \cite[Proposition 4.7.1]{Berts1} for each local problem \eqref{ADAL_1} imply that for any $\bx_i \in \mathcal{X}_i$
\begin{align*}
0\leq  \Big\langle \bs_i^k  +  \ba_i^\top \Big[ \bl^k + \rho  \big(\ba_i\hat{\bx}_i^k + \sum_{j\neq i}\ba_j \bx_j^k -\bb\big)\Big], \bx_i-\hat{\bx}_i^k  \Big\rangle .
\end{align*}
By letting $\bx_i = \bx_i^*$ and substituting $\bl^k$ with $\hat{\bl}^k:=\bl^k+\rho \br(\hat{\bx}^k)$ in the above, we get
\begin{align}\label{1}
0\leq  \Big\langle \bs_i^k  +  \ba_i^\top \Big[ \hat{\bl}^k + \rho  \sum_{j\neq i}\ba_j ( \bx_j^k - \hat{\bx}_j^k )\Big], \bx_i^*-\hat{\bx}_i^k  \Big\rangle .
\end{align}
By the definition of $\bs_i^k$, we have the relation
\begin{align}\label{subgr}
f_i(\bx_i^*) - f_i(\hat{\bx}_i^k) ~\geq~ \left\langle \bs_i^k, \bx_i^* - \hat{\bx}_i^k \right\rangle.
\end{align}
Substituting this into \eqref{1}, we get
\begin{align*}
& f_i(\bx_i^*) - f_i(\hat{\bx}_i^k) ~+~  \Big\langle \hat{\bl}^k , \ba_i \big(\bx_i^*-\hat{\bx}_i^k \big) \Big\rangle  \\
& \qquad +~ \rho \Big\langle \ba_i\big(  \bx_i^*-\hat{\bx}_i^k \big) ,  \sum\nolimits_{j\neq i} \ba_j \big( \bx_j^{k} - \hbx_j^k\big) \Big\rangle   ~\geq~ 0.\notag
\end{align*}
Summing over all $i$, we get
\begin{align*}
& F(\bx^*) - F(\hat{\bx}^k) ~+~  \Big\langle \hat{\bl}^k , \sum\nolimits_i\ba_i \big(\bx_i^*-\hat{\bx}_i^k \big) \Big\rangle  \\
&~+~ \rho \sum\nolimits_i \Big\langle \ba_i\big(\bx_i^*-\hat{\bx}_i^k\big) , \sum\nolimits_{j\neq i}\ba_j \big( \bx_j^k- \hat{\bx}_j^k \big)\Big\rangle  ~\geq~ 0.
\end{align*}
Substituting $ \sum_i \ba_i\big(\bx_i^*-\hat{\bx}_i^k\big) =  \bb - \sum_i \ba_i  \hat{\bx}_i^k = -\br(\hbx^k)$, adding and subtracting $\langle\bl, \br(\hbx^k) \rangle$, and rearranging terms in the above inequality  we get
\begin{align*}
&   -\Big\langle \hat{\bl}^k-\bl , \br(\hbx^k) \Big\rangle   ~\geq~  F(\hat{\bx}^k) - F(\bx^*) + \langle\bl, \br(\hbx^k) \rangle\\
&\qquad+~ \rho \sum\nolimits_{i} \Big\langle \ba_i\big(\hat{\bx}_i^k - \bx_i^*\big) , \sum\nolimits_{j\neq i}\ba_j \big( \bx_j^k- \hat{\bx}_j^k \big)\Big\rangle. \notag
\end{align*}
To avoid cluttering the notation, we temporarily disregard the term $F(\hat{\bx}^k) - F(\bx^*) + \langle\bl, \br(\hbx^k) \rangle$, i.e., we consider only the terms
\begin{align*}
&   -\Big\langle \hat{\bl}^k-\bl , \br(\hbx^k) \Big\rangle   ~\geq~ \\
&\qquad\rho \sum\nolimits_{i} \Big\langle \ba_i\big(\hat{\bx}_i^k - \bx_i^*\big) , \sum\nolimits_{j\neq i}\ba_j \big( \bx_j^k- \hat{\bx}_j^k \big)\Big\rangle. \notag
\end{align*}
Add the term ${\rho}\sum_i\big\langle \ba_i( \hbx_i^k -\bx_i^*) , \ba_i(\bx_i^k-\hbx_i^k) \big\rangle$ to both sides of the above inequality, and group the terms on the right-hand side by their common factor to get
\begin{align}\label{2}
& {\rho}\sum\nolimits_i\Big\langle \ba_i( \hbx_i^k -\bx_i^*) , \ba_i(\bx_i^k-\hbx_i^k) \Big\rangle - \Big\langle \hat{\bl}^k-\bl ,  \br(\hbx^k) \Big\rangle \nonumber\\
&  \geq~ {\rho}\sum\nolimits_i\Big\langle \ba_i( \hbx_i^k -\bx_i^*) , \sum\nolimits_j \ba_j(\bx_j^k-\hbx_j^k) \Big\rangle.\notag\\
& =~ {\rho}\Big\langle \br(\hbx^k)  , \br(\bx^k) - \br(\hbx^k) \Big\rangle,
\end{align}
where the last equality is from $\sum_j\ba_j(\bx_j^k-\hat{\bx}_j^k) = \br(\bx^k)-\br(\hat{\bx}^k)$.
Next, we represent
\begin{gather*}
\ba_i\hat{\bx}_i^k-\ba_i\bx_i^* = (\ba_i \bx_i^k-\ba_i \bx_i^*)+(\ba_i\hat{\bx}_i^k-\ba_i\bx_i^k) ~~\text{and}\\
{\hat{\bl}^k-\bl = (\bl^k-\bl)+(\hat{\bl}^k-\bl^k)} = (\bl^k-\bl) + \rho \br(\hbx^k),
\end{gather*}
in the left-hand side of \eqref{2} to obtain
\begin{align*}
& {\rho}\sum\nolimits_i\Big\langle \ba_i( \bx_i^k -\bx_i^*) , \ba_i(\bx_i^k-\hbx_i^k) \Big\rangle - \Big\langle \bl^k-\bl ,  \br(\hbx^k) \Big\rangle  \notag\\
& ~~~~\geq~  {\rho} \sum\nolimits_i \|\ba_i(\bx_i^{k}-\hat{\bx}_i^k)\|^2 ~+~ {\rho}\| \br(\hbx^k) \|   ^2 \\
& \qquad\qquad\qquad\qquad~~~+~ {\rho}\Big\langle \br(\hbx^k)  , \br(\bx^k) - \br(\hbx^k) \Big\rangle.  \nonumber
\end{align*}
Adding $ -(1-\tau)\rho\big\langle \br(\bx^k),\br(\hat{\bx}^k)\big\rangle$ to both sides of the above inequality and recalling the definition of $\bar{\bl}^k$ in \eqref{lambda_bar}, we get
\begin{align}\label{4}
& {\rho}\sum\nolimits_i\Big\langle \ba_i( \bx_i^k -\bx_i^*) , \ba_i(\bx_i^k-\hbx_i^k) \Big\rangle - \Big\langle \bar{\bl}^k-\bl ,  \br(\hbx^k) \Big\rangle  \notag\\
&~ \geq~  {\rho} \sum\nolimits_i \|\ba_i(\bx_i^{k}-\hat{\bx}_i^k)\|^2 ~+~ {\rho}\| \br(\hbx^k) \|   ^2 \\
&  ~~~+~ {\rho}\Big\langle \br(\hbx^k)  , \br(\bx^k) - \br(\hbx^k) \Big\rangle  -(1-\tau)\rho\Big\langle \br(\bx^k),\br(\hat{\bx}^k)\Big\rangle.  \nonumber
\end{align}
Considering only the last two terms on the right hand side of \eqref{4}, we can write
\begin{align}\label{6}
& {\rho}\Big\langle \br(\hbx^k)  , \br(\bx^k) - \br(\hbx^k) \Big\rangle  - (1-\tau)\rho\Big\langle \br(\bx^k),\br(\hat{\bx}^k)\Big\rangle  \nonumber\\
&=~{\rho}\Big\langle \br(\hbx^k)  , \br(\bx^k) - \br(\hbx^k) \Big\rangle   \\\nonumber
&\qquad \qquad - (1-\tau)\rho\Big\langle \br(\bx^k)- \br(\hbx^k) +\br(\hbx^k),\br(\hat{\bx}^k)\Big\rangle  \\\nonumber
&=~ \tau\rho\Big\langle \br(\hbx^k)  , \br(\bx^k) - \br(\hbx^k) \Big\rangle  -(1-\tau)\rho\|\br(\hat{\bx}^k)\|^2.
\end{align}
%
%
%
We now consider the last term of the above equality. Each one of the summands in this term is bounded below by
\begin{align*}
& \tau\rho\Big\langle \br(\hbx^k)  ,  \ba_i(\bx_i^k-\hat{\bx}_i^k) \Big\rangle = \tau\rho\sum_{j=1}^m  \big[\br(\hbx^k)\big]_j  \big[ \ba_i( \bx_i^k-\hat{\bx}_i^k )\big]_j \\
&~\geq~ -\frac{1}{2} \sum\nolimits_{j=1}^m \Big(\rho\Big[\ba_i(\bx_i^k-\hat{\bx}_i^k)\Big]_j^2 ~+~ \tau^2\rho\big[\br(\hbx^k)\big]_j^2\Big).
\end{align*}
Note, however, that some of the rows of $\ba_i$ might be zero. If $[\ba_i]_j = \mathbf{0}$, then it follows that $ \big[\br(\hbx^k)\big]_j  \big[ \ba_i( \bx_i^k-\hat{\bx}_i^k )\big]_j = 0$.   Hence,  denoting the set of nonzero rows of $\ba_i$ as $\mathcal{Q}_i$, i.e., $\mathcal{Q}_i = \{ j = 1,\dots,m : [\ba_i ]_j \neq \mathbf{0} \}$, we can obtain a tighter lower bound for each $\tau\rho\Big\langle \br(\hbx^k)  ,  \ba_i(\bx_i^k-\hat{\bx}_i^k) \Big\rangle$ term as
\begin{align*}
&\tau\rho\Big\langle \br(\hbx^k)  ,  \ba_i(\bx_i^k-\hat{\bx}_i^k) \Big\rangle  ~\geq~ \\
&-\frac{1}{2} \sum\nolimits_{j\in\mathcal{Q}_i} \Big(\rho\Big[\ba_i(\bx_i^k-\hat{\bx}_i^k)\Big]_j^2 ~+~ {\tau^2}{\rho}\big[\br(\hbx^k)\big]_j^2\Big).\notag
\end{align*}
Recalling that $q$ denotes the maximum number of non-zero blocks $[ \ba_i ]_j$ over all $j$, and summing inequality \eqref{6} over all $i$, we observe that each quantity $[\br(\hbx^k)]_j^2$ is included in the summation at most $q$ times. This leads us to the bound
\begin{align}\label{7}
& \tau\rho \Big\langle \br(\hbx^k)  ,  \br(\bx^k) - \br(\hbx^k) \Big\rangle  \notag\\
&\quad ~=~  \sum\nolimits_i\tau\rho\Big\langle \br(\hbx^k)  ,  \ba_i(\bx_i^k-\hat{\bx}_i^k) \Big\rangle\\
& \quad~\geq~ -  \frac{\rho}{2}\sum\nolimits_i  \| \ba_i(\bx_i^k-\hat{\bx}_i^k) \|^2 ~-~ \frac{\tau^2q\rho}{2}\| \br(\hbx^k) \|^2 .\notag
\end{align}
Substituting \eqref{6}-\eqref{7} back into \eqref{4}, we arrive at
\begin{align}\label{8}
& {\rho}\sum\nolimits_i\Big\langle \ba_i( \bx_i^k -\bx_i^*) , \ba_i(\bx_i^k-\hbx_i^k) \Big\rangle - \Big\langle \bar{\bl}^k-\bl ,  \br(\hbx^k) \Big\rangle  \notag\\
& \geq  \frac{\rho}{2} \sum\nolimits_i \|\ba_i(\bx_i^{k}-\hat{\bx}_i^k)\|^2 + \rho(\tau-\frac{\tau^2q}{2}) \| \br(\hbx^k) \|   ^2.
\end{align}
Recall that until now we have disregarded the term $F(\hat{\bx}^k) - F(\bx^*) + \langle\bl, \br(\hbx^k) \rangle$. Reinstating this term in \eqref{8}, we get
\begin{align}\label{lemma3_eq}
&F(\hbx^k) -  F(\bx^*) + \langle \bl, \br(\hbx^k)\rangle \\
&\leq~
{\rho}\sum\nolimits_i\Big\langle \ba_i( \bx_i^k -\bx_i^*) , \ba_i(\bx_i^k-\hbx_i^k) \Big\rangle -  \Big\langle \bar{\bl}^k-\bl ,  \br(\hbx^k) \Big\rangle \notag\\
&    - \frac{\rho}{2} \sum\nolimits_i \|\ba_i(\hat{\bx}_i^k-\bx_i^{k})\|^2 - \rho(\tau-\frac{\tau^2 q}{2})\| \br(\hbx^k) \|^2.\notag
\end{align}

We now represent the right-hand side of the desired result using the definition of $\bar\bl^k$ in (\ref{lambda_bar}) and Lemma \ref{lem:lbar}. For all $k$,
we have:
\begin{align*}
& \phi^k(\bl) - \phi^{k+1}(\bl)  =  \sum\nolimits_{i=1}^N {\rho}\|\ba_i(\bx_i^k-\bx_i^*)\|^2 +\frac{1}{\rho}\|\bar{\bl}^k-\bl\|^2 \notag\\
& \qquad - \sum\nolimits_{i=1}^N {\rho}\|\ba_i(\bx_i^{k+1}-\bx_i^*)\|^2 -\frac{1}{\rho}\|\bar{\bl}^{k+1}-\bl\|^2\notag\\
&=2\tau \bigg[ {\rho}\sum\nolimits_i\Big\langle \ba_i( \bx_i^k -\bx_i^*) , \ba_i(\bx_i^k-\hbx_i^k) \Big\rangle - \notag\\
&  \Big\langle \bar{\bl}^k-\bl ,  \br(\hbx^k) \Big\rangle \bigg] - \tau^2 \bigg[\sum_i {\rho}\|\ba_i(\hat{\bx}_i^k-\bx_i^{k})\|^2 + {\rho}\| \br(\hbx^k) \|^2\bigg].\notag
\end{align*}
Rearranging terms in the above equation, we get that
\begin{align*}
&\frac{1}{2\tau}\big( \phi^k(\bl) - \phi^{k+1}(\bl) \big) \\
& \ge {\rho}\sum\nolimits_i\Big\langle \ba_i( \bx_i^k -\bx_i^*) , \ba_i(\bx_i^k-\hbx_i^k) \Big\rangle -  \Big\langle \bar{\bl}^k-\bl ,  \br(\hbx^k) \Big\rangle \notag\\
&    - \frac{\rho}{2} \sum\nolimits_i \|\ba_i(\hat{\bx}_i^k-\bx_i^{k})\|^2 - \rho(\tau-\frac{\tau^2 q}{2})\| \br(\hbx^k) \|^2   ,\notag
\end{align*}
where the last inequality follows from  $\tau \in (0,\frac{1}{q})$. Recall that $\tau$ is the stepsize parameter used in the second step of ADAL (cf. Eq. \eqref{ADAL_2}). Therefore, combining this with \eqref{lemma3_eq}, we arrive at the desired result.

\bibliographystyle{IEEEtran}
\bibliography{references1,references2}
\end{document}